\documentclass{article}
\usepackage[utf8]{inputenc}
\usepackage{amsmath,amssymb,trimclip,adjustbox}
\usepackage{amsfonts}
\usepackage{amssymb}
\usepackage{amsthm}
\usepackage{amsmath}
\usepackage{mathrsfs}
\usepackage{stmaryrd}
\newtheorem{theorem}{Theorem}[section]
\newtheorem{proposition}[theorem]{Proposition}
\newtheorem{nota}[theorem]{Notation}
\newtheorem{corollary}[theorem]{Corollary}
\newtheorem{lemma}[theorem]{Lemma}
\newtheorem{conj}[theorem]{Conjecture}
\newtheorem{mydef}[theorem]{Definition}
\newtheorem{claim}[theorem]{Claim}
\newtheorem{question}[theorem]{Question}
\newtheorem{obs}[theorem]{Observation}

\title{Small infinite partitions and other features of the Nowhere Centered ideal}
\author{Mario Jardon Santos}
\date{}

\begin{document}
\maketitle

\begin{abstract}
    The \textit{nowhere dense ideal} $\mathcal{NC}$ is introduced. It is a coanalytic ideal of $\omega\times\omega$ whose defining characteristic is that the sets of the form $X\times Y$, where $X,Y$ are infinite subsets of $\omega$, are dense in the quotient $\mathcal{P}(\omega\times\omega)\slash\mathcal{NC}$. This quotient has countable partitions and consistently has partitions of size $\omega_1$ while $\mathfrak{a}>\omega_1$. This represents a huge contrast with other definable quotients. Other combinatorial features of this ideal are presented, as well as some results on a family of similar, higher dimensional ideals. 
\end{abstract}

\section{Introduction}

A subfamily $\mathcal{I}\subseteq \mathcal{P}(X)$, the power set of some countable set $X$, is called an ideal if \begin{enumerate}
    \item $X\notin\mathcal{I}$,
    
    \item $A\cup B\in\mathcal{I}$ for all $A,B\in\mathcal{I}$, and 
    
    \item $A\in\mathcal{I}$ whenever $A\subseteq B$ and $B\in\mathcal{I}$. 
\end{enumerate} Ideals on countable sets play an important role in infinite combinatorics. In one hand we have the quotients modulo some ideal\footnote{The reader is referred to \cite{handbook} for general concepts and notations on Boolean algebras.}: If $\mathcal{I}\subseteq \mathcal{P}(X)$ is an ideal, for some countable set $X$, define on $\mathcal{P}(X)$ the equivalence relation $\sim_{\mathcal{I}}$, where $A\sim_{\mathcal{I}}B$ if their symmetric difference $A\triangle B\in\mathcal{I}$, for all $A,B\subseteq X$. Then on the set of equivalence classes $\lbrace[A]_{\mathcal{I}}\mid A\subseteq X\rbrace$ define the operations $[A]_{\mathcal{I}}\vee[B]_{\mathcal{I}}:=[A\cup B]_{\mathcal{I}}$, $[A]_{\mathcal{I}}\wedge[B]_{\mathcal{I}}:=[A\cap B]_{\mathcal{I}}$ and $-[A]_{\mathcal{I}}:=[X\setminus A]_{\mathcal{I}}$. It is not hard to see that these operations are well defined and that define a Boolean algebra on the set of equivalence classes of $\sim_\mathcal{I}$, denoted $\mathcal{P}(X)\slash\mathcal{I}$. When studying this structure the notation $[A]_\mathcal{I}$ is mostly avoided and positive elements of $\mathcal{P}(X)\slash\mathcal{I}$ are identified with the elements of $\mathcal{I}^+ :=\mathcal{P}(X)\setminus\mathcal{I}$, as well as the Boolean operations are identified with those of $\mathcal{P}(X)$. The order relation of $\mathcal{P}(X)\slash\mathcal{I}$ being important and not as directly induced by $\subseteq$ as the Boolean operations, its interpretation on $\mathcal{I}^+$ is highlighted: $A\subseteq_{\mathcal{I}}B$ iff $[A]_\mathcal{I} \leq[B]_{\mathcal{I}}$ iff $A\setminus B\in\mathcal{I}$, for all $A,B\in\mathcal{I}^+$.  

Capital among these structures is the case when $X=\omega$, $\mathcal{I}=fin:=[\omega]^{<\omega}$, and $\subseteq^* :=\subseteq_{fin}$. Many combinatorial substructures of $\mathcal{P}(\omega)\slash fin$ and their related cardinal invariants have been extrapolated to general quotients. A short list of some of them, relevant to this text, is given here.

\begin{mydef}
\label{cardinal_invariants_quotient}
Take an ideal $\mathcal{I}\subseteq \mathcal{P}(X)$, of some countable set $X$, and $A,B\in\mathcal{I}^+$.

\begin{enumerate}
    \item If $A\cap B\in\mathcal{I}$, it will be said that $A$ and $B$ are $\mathcal{I}$-almost disjoint. A subfamily of $\mathcal{I}^+$ consisting of pairwise $\mathcal{I}$-almost disjoint sets is called an $\mathcal{I}$-ad family. If said family is maximal with this property it will be called an $\mathcal{I}$-mad family or an infinite partition of $\mathcal{P}(X)\slash\mathcal{I}$. The cardinal $\mathfrak{a}(\mathcal{I})$ is defined as the smallest size of an infinite $\mathcal{I}$-mad family.
    
    \item If $\mathcal{T}\subseteq\mathcal{I}^+$ is a well-ordered substructure of $(\mathcal{I}^+ ,\supseteq_{\mathcal{I}})$ with no lower bound in $\mathcal{I}^+$, it will be called an $\mathcal{I}$-tower. The cardinal $\mathfrak{t}(\mathcal{I})$ is defined as the smallest size of an $\mathcal{I}$-tower.  
    
    \item If $A\cap B\in\mathcal{I}^+$ and $B\setminus A\in\mathcal{I}^+$, it will be said that $A$ splits $B$. The cardinal $\mathfrak{s}(\mathcal{I})$ is defined as the smallest size of a family $\mathcal{S}\subseteq\mathcal{I}^+$ such that for all $X\in\mathcal{I}^+$ there exists $Y\in\mathcal{S}$ such that $Y$ splits $X$.
    
    \item The cardinal $\mathfrak{r}(\mathcal{I})$ is defined as the smallest size of a subfamily $\mathcal{R}\subseteq\mathcal{I}^+$ such that no element of $\mathcal{I}^+$ splits all elements of $\mathcal{R}$. In this case $\mathcal{R}$ is called a reaping family. 
\end{enumerate}
\end{mydef}

All these cardinals are generalizations of the well-studied case when $\mathcal{I}=fin$. In that case no ideal is specified and the cardinal is identified by the corresponding fraktur font letter alone. For a general and basic understanding of this case the reader is referred to \cite{blass}. For results of both general and specific cases of these cardinal invariants, though mostly of $\mathfrak{a}(\mathcal{I})$, the reader is referred to \cite{steprans}, \cite{farkas} and \cite{el_juris_y_el_dilip}. For more general cases of infinite Boolean algebras a great resource is \cite{monk}. 

Another source of study for an ideal $\mathcal{I}$ on a countable set $X$ is seeing it as a combinatorial substructure of $\mathcal{P}(X)\slash[X]^{<\omega}$. There the following cardinal invariants arise.

\begin{mydef}
\label{cichon}

  Let $\mathcal{I}$ be an ideal on the countable set $X$. Define  
    
    \begin{itemize}
        \item (Additivity) $add^* (\mathcal{I}):=\min\lbrace\vert\mathcal{F}\vert\mid\mathcal{F}\subseteq\mathcal{I}~\neg\exists I\in\mathcal{I}~\forall A\in\mathcal{F}~A\nsubseteq^* I\rbrace$
        
        \item (Covering) $cov^* (\mathcal{I}):=\min\lbrace\vert\mathcal{F}\vert\mid\mathcal{F}\subseteq\mathcal{I}~\forall B\in[X]^\omega ~\exists A\in\mathcal{F}~\vert B\cap A\vert=\omega\rbrace$
        
        \item (Uniformity) $non^* (\mathcal{I}):=\min\lbrace\vert\mathcal{X}\vert\mid\mathcal{X}\subseteq[X]^\omega ~\neg\exists A\in\mathcal{I}~\forall B\in\mathcal{X}~\vert A\cap B\vert=\omega\rbrace$
        
        \item (Cofinality) $cof^* (\mathcal{I}):=\min\lbrace\vert\mathcal{F}\vert\mid\mathcal{F}\subseteq\mathcal{I}~\forall A\in\mathcal{I}~\exists B\in\mathcal{F}~A\subseteq^* B\rbrace$
    \end{itemize}
\end{mydef}

In Section \ref{quotient} of this text an ideal on $\omega\times\omega$, the \textit{Nowhere Centered ideal} $\mathcal{NC}$ is introduced and the structures of its quotient, as defined in Definition \ref{cardinal_invariants_quotient}, are studied, with great emphasis on $\mathfrak{a}(\mathcal{I})$ and infinite partitions. In section \ref{third_dimension} hints are given to higher dimensional generalizations of $\mathcal{NC}$, alongside an important result on infinite partitions on the corresponding quotients. Finally in Section \ref{add_cov_non_cof} the cardinal invariants of Definition \ref{cichon} relative to the ideal $\mathcal{NC}$ are studied. The reader is assumed to have a basic knowledge on concepts and notation of set theory and forcing. 

Since most of this work is set in $\omega\times\omega$, for $A\subseteq\omega\times\omega$ and $n<\omega$ set the notation $$A(n):=\lbrace m<\omega\mid(n,m)\in A\rbrace.$$ Also, for $X\subseteq\omega$ and $\lbrace Y_n \mid n\in X\rbrace\subseteq \mathcal{P}(\omega)$ define $$\coprod_{n\in X}Y_n :=\bigcup_{n\in X}\lbrace n\rbrace\times Y_n .$$

\section{The Nowhere Centered Ideal and its quotient}
\label{quotient}

Recall that a family $\mathcal{X}\subseteq[\omega]^\omega$ is said to be \textit{centered}, if $\bigcap F$ is infinite for all $F\in[\mathcal{X}]^{<\omega}$. From this almost natural concept in infinite combinatorics an ideal on $\omega\times\omega$ is defined. 

\begin{mydef} 
\label{nci}

The Nowhere Centered ideal is defined on $\omega\times\omega$ as follows
$$\mathcal{NC}:=\lbrace A\subseteq\omega\times\omega\mid\forall X\in[\omega]^\omega ~the~set~\lbrace A(n) : n\in X \rbrace~is~not~centered\rbrace.$$
\end{mydef}

Obviously $\omega\times\omega\notin\mathcal{NC}$ and $\mathcal{NC}$ is closed under subsets. Take $A_0 , A_1 \subseteq\omega\times\omega$ and suppose that $A_0 \cup A_1 \notin\mathcal{NC}$. There exists $X\in[\omega]^\omega$ such that $\lbrace A_0 (n)\cup A_1 (n)\mid n\in X\rbrace$ is a centered family. If $\mathcal{U}$ is an ultrafilter extending this family, there exist $i<2$ and $X^\prime \in[X]^\omega$ such that $\lbrace A_i (n)\mid n\in X^\prime \rbrace$ is subset of $\mathcal{U}$, and hence is a centered family. Therefore $A_i \notin\mathcal{NC}$, and we conclude that $\mathcal{NC}$ is closed under finite unions, and hence an ideal. 

\begin{obs}
\label{nowhere}
If $X\in[\omega]^\omega$ and $A\in\mathcal{NC}$, then there exist $n<\omega$ and $Y\in[X]^\omega$ such that $\vert A(m)\cap Y\vert<\omega$, for all $m\geq n$. 
\end{obs}

Indeed we claim that there exists $n<\omega$ such that $\vert X\setminus\bigcup_{m\in F}A(m)\vert=\omega$, for all $F\in[\omega\setminus n]^{<\omega}$. Otherwise for all $n<\omega$ there exists $F\in[\omega\setminus n]^{<\omega}$ such that $X\subseteq^* \bigcup_{m\in F}A(m)$. But hence any ultrafilter containing $X$ would give us  $B\in[\omega]^\omega$ such that $\lbrace A(m)\mid m\in B\rbrace$ is a centered family, which is a contradiction. Given such $n<\omega$, getting such $Y\in[X]^\omega$ is an easy task.

One could see Observation \ref{nowhere} as a \textit{literary justification} for the naming of the ideal $\mathcal{NC}$, since it has a similar flavour as the observation that whenever we have $\sigma\in\omega^{<\omega}$ and a subtree $T\subseteq\omega^{<\omega}$ whose branches form a nowhere dense set of $\omega^\omega$, then there exists $\tau\in\omega^{<\omega}$, an extension of $\sigma$, such that $\langle\tau\rangle\cap T=\emptyset$. In this case given the rectangle $\omega\times X$ and $A\in\mathcal{NC}$ we get an ``extension" of the rectangle which ``avoids" $A$. A measure of this avoidance as well as a known face to compare $\mathcal{NC}$ to is given by the following ideal.

\begin{mydef}
 $fin\times fin:=\lbrace A\subseteq\omega\times\omega\mid\exists n<\omega~\forall m\geq n~\vert A(m)\vert<\omega\rbrace$.
\end{mydef}

 The ideal $fin\times fin$ is often thought of as the ideal on $\omega\times\omega$ generated by the columns $\lbrace n\rbrace\times\omega$, for $n<\omega$, and the sets below a function $\lbrace(n,m)\in\omega\times\omega\mid m\leq f(n)\rbrace$, for $f\in\omega^\omega$. Since for $A\in fin\times fin$ and $X\in[\omega]^\omega$ there exists $n$ such that $$\bigcap_{i\in X\cap n} A(i)<\omega,$$ it easily follows that $fin\times fin\subseteq\mathcal{NC}$. This inequality is strict. Indeed, if $\lbrace Y_n \mid n<\omega\rbrace\subseteq [\omega]^\omega$ is an almost disjoint family, the set $$Y:=\coprod_{n<\omega}Y_n$$ is an element of $\mathcal{NC}$ which does not lie in $fin\times fin$.

For getting to know the quotient $\mathcal{P}(\omega\times\omega)\slash\mathcal{NC}$ take $A\in\mathcal{NC}^+$. By definition there exists $X\in[\omega]^\omega$ such that $$\lbrace A(n)\mid n\in X\rbrace$$ is a centered family. If $Y\in[\omega]^\omega$ is a pseudointersection of said family, it follows that $$X\times Y\subseteq_{fin\times fin} A,$$ which means that $$\qquad\qquad\qquad X\times Y\subseteq_{\mathcal{NC}} A. \qquad \qquad \qquad (\star)$$ Observe that $Y$ is as much a key in determining $A\in\mathcal{NC}^+$ as $X$ is. Therefore $\mathcal{NC}^+$ is the second coordinate projection of the subset of $[\omega]^\omega \times \mathcal{P}(\omega\times\omega)$ consisting of all $(Y,A)$ satisfying the Borel formula $$\forall n<\omega\exists m\geq n\exists k<\omega\forall l\geq k(l\in Y\Rightarrow l\in A(m)),$$ which means there exist infinite $m<\omega$ such that $Y\subseteq^* A(m)$. It follows that the ideal $\mathcal{NC}$ is coanalytic.\footnote{Since $\mathcal{P}(\omega\times\omega)$ has a natural topology induced by that of the Polish space $2^{\omega\times\omega}$, its ideals can be studied by its topological properties. For elements of descriptive set theory see \cite{kechris}.}

From $(\star)$ it also follows that the set $$\qquad\qquad\lbrace[X\times Y]_{\mathcal{NC}}\mid X,Y\in[\omega]^\omega\rbrace \qquad\quad \circledast$$ is dense in $\mathcal{P}(\omega\times\omega)\slash\mathcal{NC}$. This important fact will be heavily used all through this section. This means that $\mathcal{P}(\omega)\slash fin\oplus \mathcal{P}(\omega)\slash fin$ is (isomorphic to) a dense subalgebra of $\mathcal{P}(\omega\times\omega)\slash\mathcal{NC}. $\footnote{
$\mathcal{P}(\omega)\slash fin\oplus \mathcal{P}(\omega)\slash fin$ is isomorphic to the clopen sets of $(\beta\omega\setminus\omega)^2$, where $\beta\omega\setminus\omega$ is the space of free ultrafilters on $\omega$. Its non-zero elements are of the form $\bigcup_{i<n}\widetilde{X}_i \times\widetilde{Y}_i$, where each $\widetilde{X}_i$ and $\widetilde{Y}_i$ are the sets of free ultrafilters containing $X_i$ and $Y_i$, some infinite subsets of $\omega$. The very conception of $\mathcal{NC}$ derives from the search for a quotient approximating this free product, and so this relation is highlighted. All cardinal invariants like $\mathfrak{a}(2)$ in this text refer to cardinal invariants of this Boolean algebra.} It is also easy to see that $$\lbrace[\omega\times X]_{\mathcal{NC}}\mid X\in[\omega]^\omega \rbrace$$ is a (regular) subalgebra of $\mathcal{P}(\omega\times\omega)\slash\mathcal{NC}$ isomorphic to $\mathcal{P}(\omega)\slash fin$. 

There are several combinatorial ways in which the Boolean algebras $\mathcal{P}(\omega)\slash fin$, $\mathcal{P}(\omega)\slash fin\oplus \mathcal{P}(\omega)\slash fin$ and $\mathcal{P}(\omega\times\omega)\slash\mathcal{NC}$ are similar. 

\begin{theorem}
$\mathfrak{s}=\mathfrak{s}(2)=\mathfrak{s}(\mathcal{NC})$.
\end{theorem}

\begin{proof}
The first equality holds because $\mathfrak{s}(A\oplus B)=\min\lbrace\mathfrak{s}(A),\mathfrak{s}(B)\rbrace$ for all non-atomic Boolean algebras $A$ and $B$. This fact can be found in Theorem 10 of \cite{mio}. For proving $\mathfrak{s}\leq\mathfrak{s}(\mathcal{NC})$ first take $\kappa<\mathfrak{s}$ and $\lbrace A_\alpha \mid\alpha<\kappa\rbrace\subseteq\mathcal{NC}^+ $. Since the family $$\lbrace A_\alpha (n)\mid\alpha<\kappa,~n<\omega\rbrace$$ is not splitting in $\mathcal{P}(\omega)\slash fin$, take $Y\in[\omega]^\omega$ witnessing this fact. Define $$K:=\lbrace\alpha<\kappa\mid\exists^{\infty}n<\omega~(Y\subseteq^* A_\alpha (n))\rbrace.$$ For each $\alpha\in K$ consider the set $X_\alpha :=\lbrace n<\omega\mid Y\subseteq^* A_\alpha (n)\rbrace$. The family $\lbrace X_\alpha \mid\alpha\in K\rbrace$ is not splitting in $\mathcal{P}(\omega)\slash fin$. Take $X\in[\omega]^\omega$ witnessing this fact. The set $X\times Y$ witnesses that  $\lbrace A_\alpha \mid\alpha<\kappa\rbrace\subseteq\mathcal{NC}^+ $ is not splitting in $\mathcal{P}(\omega\times\omega)\slash\mathcal{NC}$. Indeed, take $\alpha<\kappa$. If $\alpha\in K$ and $X\subseteq^* X_\alpha$, then $X\times Y\subseteq_{\mathcal{NC}}A_\alpha$. If $\alpha\in K$ and $\vert X\cap X_\alpha \vert<\omega$, then $\vert Y \cap A_\alpha (n)\vert <\omega$ for almost all $n\in X$. On the other hand if $\alpha\notin K$, then $\vert Y \cap A_\alpha (n)\vert <\omega$ for almost all $n<\omega$. From both last cases it follows that $(X\times Y)\cap A_\alpha \in\mathcal{NC}$. We conclude that $\mathfrak{s}\leq\mathfrak{s}(\mathcal{NC})$.

Now take family $\lbrace X_\alpha \mid\alpha<\mathfrak{s}\rbrace\subseteq[\omega]^\omega$ splitting in $\mathcal{P}(\omega)\slash fin$. The family $$\mathcal{S}:=\lbrace X_\alpha \times X_\beta \mid\alpha,~\beta<\mathfrak{s}\rbrace$$ is splitting in $\mathcal{P}(\omega\times\omega)\slash\mathcal{NC}$. Take $A\in\mathcal{NC}^+$, which will be proved to be split by some element of $\mathcal{S}$. Since the set $\circledast$ is dense, w.l.o.g $A=X\times Y$ for some $X,~Y\in[\omega]^\omega$. Take $\alpha<\mathfrak{s}$ and $\beta<\mathfrak{s}$ such that $X_\alpha$ splits $X$ and $X_\beta$ splits $Y$. It follows that both $(X\times Y)\cap(X_\alpha \times X_\beta)\in\mathcal{NC}^+$ and $(X\times Y)\setminus(X_\alpha \times X_\beta)\in\mathcal{NC}^+$, i.e. $X_\alpha \times X_\beta$ splits $A$. We conclude that $\mathfrak{s}(\mathcal{NC})\leq\mathfrak{s}$. 
 
\end{proof}

Concerning reaping families in $\mathcal{P}(\omega\times\omega)\slash\mathcal{NC}$ matters are not as straightforward. Nevertheless non surprising bounds were found for $\mathfrak{r}(\mathcal{NC})$. A family $\mathcal{R}\subseteq[\omega]^\omega$ is said to be $\sigma$-\textit{reaping} if for all $\lbrace X_n \mid n<\omega\rbrace\subseteq[\omega]^\omega$ there exists $A\in\mathcal{R}$ such that for all $n<\omega$ either 
$A\subseteq^* X_n$ or $\vert A\cap X_n \vert<\omega$. The cardinal $\mathfrak{r}_\sigma$ is defined to be the smallest cardinality of a $\sigma$-reaping family. 

\begin{theorem}
$\mathfrak{r}\leq\mathfrak{r}(2)\leq\mathfrak{r}(\mathcal{NC})\leq\mathfrak{r}_\sigma$. 
\end{theorem}

\begin{proof}
Take $\kappa<\mathfrak{r}$ and $\lbrace A_{\alpha} ,B_{\alpha}\mid\alpha<\kappa\rbrace\subseteq[\omega]^\omega$. Take $X,Y\in[\omega]^\omega$ such that $X$ splits $A_\alpha$ and $Y$ splits $B_\alpha$, for all $\alpha<\kappa$. Therefore $\widetilde{X}\times\widetilde{Y}$ splits (as clopen subset of $(\beta\omega\setminus\omega)^2$) $\widetilde{A}_\alpha \times\widetilde{B}_\alpha$, for all $\alpha<\kappa$. We conclude that $\kappa<\mathfrak{r}(2)$.

Take $\kappa<\mathfrak{r}(2)$ and $\lbrace A_\alpha \mid\alpha<\kappa\rbrace\subseteq\mathcal{NC}^+$. For $\alpha<\kappa$ take $X_\alpha ,~Y_\alpha \in[\omega]^\omega$ such that $(X_\alpha \times Y_\alpha )\subseteq_{\mathcal{NC}}A_\alpha$. There exist $C_0 ,...,C_{n-1},D_0 ,...,D_{n-1}\in[\omega]^\omega$, for some $n<\omega$, such that $$\widetilde{X}_\alpha \times\widetilde{Y}_\alpha \cap\bigcup_{i<n}\widetilde{C}_i \times\widetilde{D}_i\neq\emptyset\neq\widetilde{X}_\alpha \times\widetilde{Y}_\alpha \setminus\bigcup_{i<n}\widetilde{C}_i \times\widetilde{D}_i,$$ for all $\alpha<\kappa$. Therefore $$X_\alpha \times Y_\alpha \cap\bigcup_{i<n}C_i \times D_i \in\mathcal{NC}^+$$ and $$X_\alpha \times Y_\alpha \setminus\bigcup_{i<n} C_i \times D_i \in\mathcal{NC}^+,$$ for all $\alpha<\kappa$. We conclude that $\kappa<\mathfrak{r}(\mathcal{NC}) $.

Let $\mathcal{R}=\lbrace A_\alpha \mid\alpha<\mathfrak{r}_\sigma \rbrace\subseteq[\omega]^\omega$ be a $\sigma$-reaping family. Consider the family $\lbrace A_\alpha \times A_\beta \mid\alpha,\beta<\mathfrak{r}_\sigma \rbrace$ and take $B\in\mathcal{NC}^+$. Since infinitely many elements of $\lbrace B(n)\mid n<\omega\rbrace$ are infinite, there exists $\alpha<\mathfrak{r}_\sigma$ such that for all $n<\omega
$ either $A_\alpha \subseteq^* B(n)$ or $\vert A_\alpha \cap B(n)\vert<\omega$. Define $X:=\lbrace n<\omega \mid A_\alpha \subseteq^* B(n)\rbrace$. Since $\mathcal{R}$ is a reaping family, there exists $\beta<\mathfrak{r}_\sigma$ such that either $A_\beta \subseteq^* X$ or $\vert A_\beta \cap X\vert<\omega$. In the first case it follows that $A_\beta \times A_\alpha \subseteq_{\mathcal{NC}}B$; in the second one, that $(A_\beta \times A_\alpha)\cap B\in\mathcal{NC}$. In either case we conclude that $\lbrace A_\alpha \times A_\beta \mid\alpha,\beta<\mathfrak{r}_\sigma \rbrace$ is a reaping family in $\mathcal{P}(\omega\times\omega)\slash\mathcal{NC}$. Therefore $\mathfrak{r}(\mathcal{NC})\leq\mathfrak{r}_\sigma$. 
\end{proof}

Notwithstanding these results, considering towers and infinite partitions huge differences between the corresponding cardinal invariants of $\mathcal{P}(\omega)\slash fin$ and $\mathcal{P}(\omega\times\omega)\slash\mathcal{NC}$ begin to arise. On one hand any infinite partition of   $\mathcal{P}(\omega)\slash fin$ induces a partition of  $\mathcal{P}(\omega)\slash fin\oplus \mathcal{P}(\omega)\slash fin$, whose infinite partitions induce infinite partitions on $\mathcal{P}(\omega\times\omega)\slash\mathcal{NC}$. It follows that $$\mathfrak{a}(\mathcal{NC})\leq\mathfrak{a}(2)\leq\mathfrak{a}.$$ Nevertheless the left inequality is strict. This stems from the fact that the all countable subsets of $$\lbrace[\omega\times X]_{\mathcal{NC}}\mid X\in[\omega]^\omega \rbrace$$ 
have a supremum in $\mathcal{P}(\omega\times\omega)\slash\mathcal{NC}$. 

\begin{lemma}
\label{aI=omega}
If $\lbrace Y_n : n<\omega\rbrace\subseteq[\omega]^\omega ,$ then 

$$\coprod_{n<\omega}\bigcup_{i\leq n} Y_i = \bigvee_{n<\omega} \omega\times Y_n .$$
\end{lemma}

\begin{proof}
It is immediate that $$\omega\times Y_n \subseteq_{\mathcal{NC}}\coprod_{n<\omega}\bigcup_{i\leq n}Y_i$$ for all $n<\omega$. Take $A\times B\subseteq_{\mathcal{NC}}\coprod_{n<\omega}\bigcup_{i\leq n}Y_i$, with $A,B\in[\omega]^\omega$. If $B\nsubseteq^* \bigcup_{i\leq n}Y_i$ for all $n<\omega$, then there exists $B^\prime \in[B]^\omega$ such that $\vert B^\prime \cap Y_n \vert<\omega$ for all $n<\omega$, but this is clearly a contradiction. Hence there exists $n<\omega$ such that $B\subseteq^* \bigcup_{i\leq n}Y_i$. It follows that $$A\times B\subseteq_{\mathcal{NC}}\bigcup_{i\leq n}\omega\times Y_i .$$ Since every lower bound of $\coprod_{n<\omega}\bigcup_{i\leq n}Y_i$ is met by $\omega\times Y_n$ for some $n<\omega$, the lemma follows.

\end{proof}

It analogously follows that for any countable family $\lbrace X_n \mid n<\omega\rbrace\subseteq[\omega]^\omega$, the set $\coprod_{n<\omega}\bigcap_{i\leq n}X_i$ is the infimum of the family $\lbrace \omega\times X_n \mid n<\omega\rbrace$. Observe that also if $\lbrace X_n \mid n<\omega\rbrace\subseteq[\omega]^\omega$ is an increasing (decreasing) family the set $\coprod_{n<\omega}X_n$ is the supremum (infimum) of the family $\lbrace \omega\times X_n \mid n<\omega\rbrace$.

There are several other instances where a countable family $\lbrace A_n \mid n<\omega\rbrace\subseteq\mathcal{NC}^+$ has supremum in $\mathcal{P}(\omega\times\omega)\slash\mathcal{NC}$, e.g. if there exists $\lbrace X_n \mid n<\omega\rbrace$ a partition of $\omega$ such that $A_n (m)\subseteq^* X_n$ for all $n,m<\omega$. It is also worth noticing, as it will be proved later, that not all countable subfamilies of $\mathcal{P}(\omega\times\omega)\slash\mathcal{NC}$ have supremum. Nevertheless Lemma \ref{aI=omega} proves to be strong enough for  getting interesting results. 

\begin{corollary}
$\mathfrak{a}(\mathcal{NC})=\omega$.
\end{corollary}

\begin{proof}
Take $\lbrace A_n \mid n<\omega\rbrace\subseteq[\omega]^\omega$ a partition of $\omega$. Then $$\lbrace\omega\times A_n \mid n<\omega\rbrace\cup\lbrace\omega\times\omega\setminus\bigvee_{n<\omega}\omega\times A_n \rbrace$$ is clearly an $\mathcal{NC}$-mad family witnessing this statement.
\end{proof}

Furthermore we have that $\mathfrak{t}(\mathcal{NC})=\omega$. The equality $\mathfrak{a}(\mathcal{I})=\omega$ not being an uncommon result when $\mathcal{I}$ is a definable ideal on a countable set $X$, as can be seen in Chapter 2 of \cite{farkas}, the cardinal $\overline{\mathfrak{a}}(\mathcal{I})$, defined as the smallest size of an uncountable $\mathcal{I}$-mad family, was introduced. On one hand the study of this cardinal helps deepen the understanding of the combinatorics $\mathcal{P}(X)\slash\mathcal{I}$. On the other hand, be it equal to $\mathfrak{a}(\mathcal{I})$ or not, it often reflects some of the behaviour of $\mathfrak{a}$ (see \cite{farkas}). The next step in this work is to study the cardinal invariant $\overline{\mathfrak{a}}(\mathcal{NC})$. On that matter we have the following result. 

\begin{theorem}
\label{t=omega_1}
If $\mathfrak{t}=\omega_1$, then $\overline{\mathfrak{a}}(\mathcal{NC})=\omega_1$.
\end{theorem}

\begin{proof}
Suppose that $\mathfrak{t}=\omega_1$ and that $\lbrace X_\alpha :\alpha <\omega_1 \rbrace\subseteq[\omega]^\omega$ is an increasing tower of $\mathcal{P}(\omega)\slash fin$ (i.e. the complements of its elements form a tower). For $\alpha <\omega_1 $ define $$A_\alpha :=\omega\times X_\alpha \setminus\bigvee_{\beta<\alpha}\omega\times X_\beta .$$ 

These sets are well defined from lemma \ref{aI=omega} because each $\alpha$ is a countable ordinal. From its very definition it follows that  $\lbrace A_\alpha :\alpha<\omega_1 \rbrace$ is an $\mathcal{NC}$-ad family. Take $A,B\in[\omega]^\omega$ and let $\alpha$ be the least ordinal number less than $\omega_1$ such that $\vert B\cap X_\alpha \vert =\omega$. It follows that $(A\times B)\cap(\omega\times X_\alpha) \in\mathcal{NC}^+$ and that $(A\times B)\cap(\omega\times X_\beta )\in\mathcal{NC}$ for all $\beta <\alpha$. Therefore $(A\times B)\cap A_\alpha \in\mathcal{NC}^+$. Then we have an $\mathcal{NC}$-mad family of size $\omega_1$.
\end{proof}

Recall that $\mathfrak{b}$, is the smallest size of a family $\mathcal{F}\subseteq\omega^\omega$ such that for all $g\in\omega^\omega$ there exists $f\in\mathcal{F}$ such that $f(m)> g(m)$, for infinitely many $m<\omega$ ($f\nleq^* g$, in standard notation).  It is known that adding $\kappa$ many Hechler reals to a GCH model, for an uncountable regular cardinal $\kappa$, we get a model satisfying $\mathfrak{t}=\omega_1$ and $\mathfrak{b}=\kappa=\mathfrak{c}$ (see \cite{baumgartner} or \cite{brendle}). On the other hand one can prove that $\mathfrak{b}\leq\mathfrak{a}(2)$ (see \cite{note_on_products}). From this we get the following consistency result.

\begin{corollary}
\label{small_uncountable_}
If $\kappa$ is an uncountable regular cardinal of a model $V$ of ZFC, there exists a $ccc$ generic extension of $V$ where $\overline{\mathfrak{a}}(\mathcal{NC})=\omega_1 <\mathfrak{b}=\mathfrak{a}(2)=\mathfrak{a}=\kappa$. 
\end{corollary}

Since $\mathfrak{b}\leq\mathfrak{a}(fin\times fin)$, $\mathfrak{b}\leq\overline{\mathfrak{a}}(\mathcal{I})$, if $\mathcal{I}$ is an analytic $P$-ideal\footnote{An ideal $\mathcal{I}$ is a $P$-ideal iff $add^* (\mathcal{I})>\omega$.}, and $add(\mathcal{M})\leq\overline{\mathfrak{a}}(nwd)$,\footnote{$add(\mathcal{M})$ is the smallest size of a family of meager sets of the real line whose union is not a meager set. It is also equal to $\kappa$ in this generic extension.} where $nwd$ is the ideal of nowhere dense subsets of the rationals, 
these cardinal invariants are also equal to $\kappa$ in the model of Corollary \ref{small_uncountable_}. All this signals a huge contrast between the combinatorial nature of $\mathcal{P}(\omega\times\omega)\slash\mathcal{NC}$ and that of other definable quotients, as well to that of its subalgebras $\mathcal{P}(\omega)\slash fin$ and $\mathcal{P}(\omega)\slash fin\oplus \mathcal{P}(\omega)\slash fin$.


Other small substructure is given in the following theorem. Since it depends on the existence of \textit{Aronszajn trees}, a note on trees and their notation will be useful for the remaining of this section. 

If $\kappa$ and $\lambda$ are infinite ordinals, $T\subseteq\kappa^{<\lambda}$ will be called a (sub)tree if $\sigma\upharpoonright\alpha\in T$, for all $\sigma\in T$ and $\alpha\in dom(\sigma)$. If $\alpha<\lambda$, the $\alpha$-th level of $T$ is the set $$T_\alpha :=\lbrace\sigma\in T\mid dom(\sigma)=\alpha\rbrace.$$ Accordingly the restriction to $\alpha$ is the subtree of $\kappa^{<\alpha}$ $$T_{<\alpha}:=\lbrace\sigma\in T\mid dom(\sigma)<\alpha\rbrace.$$ The (cofinal) \textit{branches} of $T$ are the elements of $$[T]:=\lbrace f\in\kappa^\lambda \mid\forall\alpha<\lambda~f\upharpoonright\alpha\in T\rbrace.$$ We will say that $T$ is \textit{well-pruned} if for all $\sigma\in T$ and all $dom(\sigma)<\alpha<\lambda$ there exists $\tau\in T_{\alpha}$ such that $\sigma\subseteq\tau$.\footnote{For the purposes of this text this definition will be enough. For a more standard definition of well-pruned tree see \cite{kunen}.} A tree $T\subseteq\omega^{<\omega_1}$ is called an \textit{Aronszajn} tree if $0<\vert T_\alpha \vert\leq\omega$, for all $\alpha<\omega_1$,  and $[T]=\emptyset$. The existence of Aronszajn trees is provable in ZFC. Basic and abundant information on this subject can be found in Chapter III of \cite{kunen}.

\begin{theorem}
\label{aronszjan2}
Let $T\subseteq\omega^{<\omega_1}$ be an Aronszajn tree. Take $\lbrace A_\sigma \mid \sigma\in T\rbrace\subseteq[\omega]^\omega$ such that if $\sigma\subseteq\tau\in T$, then $A_\tau \subseteq^* A_\sigma$ and $\lbrace A_\sigma \mid \sigma\in T_\alpha \rbrace$ is a partition of $\omega$, for all $\alpha <\omega_1$. Then the sets $X_\alpha :=\bigvee_{\sigma \in T_\alpha }\omega\times A_\sigma$ for $\alpha<\omega_1$ form an $\mathcal{NC}$-tower.
\end{theorem}

\begin{proof}
Suppose there exist $A,B\in[\omega]^\omega$ such that $A\times B\subseteq_{\mathcal{NC}}X_\alpha$, for all $\alpha<\kappa$. This means that for all $\alpha$ there exists $F_\alpha \in[ T_\alpha ]^{<\omega}$ such that $$B \subseteq^* \bigcup_{\sigma\in F_\alpha} A_\sigma.$$ Extending $\lbrace\bigcup_{\sigma\in F_\alpha}A_\sigma \mid\alpha<\omega_1 \rbrace$ to an ultrafilter we can choose $\sigma_\alpha$, for all $\alpha<\omega_1$, such that $\lbrace A_{\sigma_\alpha}\mid\alpha<\omega_1 \rbrace$ is a centered family. But this means that $\lbrace\sigma_\alpha \mid\alpha<\omega_1 \rbrace$ is a branch of $T$, which is a contradiction. 
\end{proof}

While Corollary \ref{small_uncountable_} detached the ideal $\mathcal{NC}$ from many definable ideals, Theorem \ref{aronszjan2} provides it with a companion: it was proved in \cite{szymanski} that $\mathcal{P}(\omega\times\omega)\slash fin\times fin$ has a tower of size $\omega_1$.

Given this suite of results on small substructures of $\mathcal{P}(\omega\times\omega)\slash\mathcal{NC}$, one could wonder if this tower could be used in a way similar to that of proposition \ref{t=omega_1} to build a partition of size $\omega_1$, whose existence were provable in ZFC. However this is not the case for this idea. To prove this we will first prove some limits that iterating $\bigvee$ and $\bigwedge$, on countable indices, has on this Boolean algebra.  

\begin{proposition}
\label{gdelta,fsigma}
Let $T\subseteq\omega^{<\omega}$ be a subtree and take  $\lbrace A_\sigma \mid \sigma\in T\rbrace\subseteq[\omega]^\omega$ such that if $\sigma\subseteq\tau\in T$, then $A_\tau \subseteq^* A_\sigma$ and $\lbrace A_\sigma \mid \sigma\in T_n \rbrace$ is a partition of $\omega$, for all $n<\omega$. Then $$\bigwedge_{n<\omega}\bigvee_{\sigma\in T_n}\omega\times A_\sigma$$ exists in $\mathcal{P}(\omega\times\omega)\slash\mathcal{NC}$ iff there exist $T^n \subseteq T$, for $n<\omega$, finite level subtrees such that 
$$[T]=\bigcup_{n<\omega}[T^n].$$
\end{proposition}

\begin{proof}
First suppose that we have such decomposition $\lbrace T^n \mid n<\omega\rbrace$ of $T$. Define $$A:=\coprod_{n<\omega}\bigcup\lbrace A_\sigma \mid\sigma\in\bigcup_{i\leq n}T^{i}_{n}\rbrace.$$ Take a positive set $C\times B\subseteq_{\mathcal{NC}}A$. Without loss of generality, this means that $B\subseteq^* A(n)$ for all $n\in C$. It follows that for all $n\in C$ there exists $F_n \in[T_n ]^{<\omega}$ such that $B\subseteq^* \bigcup_{\sigma\in F_n} A_\sigma$, and hence that $C\times B \subseteq_{\mathcal{NC}}\bigvee_{\sigma\in T_n}\omega\times A_\sigma$ for all $n<\omega$. Therefore $A$ is a lower bound of these sets. 
 
Let $C\times B$ be a lower bound of $\bigvee_{\sigma\in T_n}\omega\times A_\sigma$, for all $n<\omega$. For all $n<\omega$ there exists $F_n \in[T_n ]^{<\omega}$ such that $B\subseteq^* \bigcup_{\sigma\in F_n}A_\sigma$. This means that there exists $f\in[T]$ such that $$\lbrace A_{f\upharpoonright n}\mid n<\omega\rbrace\cup\lbrace B\rbrace $$ is a centered family. If $f\in T^n$, this means that $\lbrace A(m)\mid m\geq n\rbrace\cup\lbrace B\rbrace$ is a centered family. Therefore $A\cap (C\times B)\in\mathcal{NC}^+$. It follows that $A=\bigwedge_{n<\omega}\bigvee_{\sigma\in T_n}\omega\times A_\sigma$.

Now suppose that there exists $A\in\mathcal{NC}^+$ such that 
$$A=\bigwedge_{n<\omega}\bigvee_{\sigma\in T_n}\omega\times A_\sigma .$$ For $F\in[\omega]^{<\omega}$ define $$A_F :=\coprod_{n\in\omega\setminus F} \bigcap_{i\in n+1\setminus F}A(i)\setminus\bigcup_{i\in F}A(i). $$

\begin{claim}
$$A= \bigvee_{F\in[\omega]^{<\omega}} A_F .$$

\end{claim}

\begin{proof} 
It is clear that $A_F \subseteq_{\mathcal{NC}}A$, for all $F\in[\omega]^{<\omega}$. Take $C\times B\subseteq_{\mathcal{NC}}A$. Without loss of generality this means that $\lbrace A(n)\mid n\in C\rbrace\cup\lbrace B\rbrace$ is a centered family. Extend $C$ to a set $C^\prime$ such that $\lbrace A(n)\mid n\in C^\prime \rbrace\cup\lbrace B\rbrace$ is a centered family, and which is maximal with this property. Let $B^\prime$ be a pseudointersection of said family and suppose that $C^\prime $ is coinfinite. 

Since $C^\prime \times B^\prime$ is a lower bound of $\bigvee_{\sigma\in T_n}\omega\times A_\sigma$ for all $n<\omega$, we can take $F_n \in[T_n ]^{<\omega}$ such that $B^\prime \subseteq^* \bigcup_{\sigma\in F_n}A_\sigma $. But it easily follows that $(\omega\setminus C^\prime )\times B^\prime$ is a positive set $\mathcal{NC}$-almost contained in each $\bigvee_{\sigma\in T_n}\omega\times A_\sigma$. Meanwhile, from the maximality of $C^\prime$, it follows that $\vert A(n)\cap B^\prime \vert<\omega$ for all $n\in\omega\setminus C^\prime$. This means that $(\omega\setminus C^\prime )\times B^\prime$ is $\mathcal{NC}$-almost disjoint to $A$, which is a contradiction. Therefore $C^\prime $ is cofinite. 

Take $F=\omega\setminus C^\prime$. It is not hard to verify that $C\times B^\prime \subseteq_{\mathcal{NC}} A_F$, whereby the claim is proved.
\end{proof}

From their definition we have that for each $\mathcal{NC}$-positive $A_F$ there exists $\lbrace B_n \mid n<\omega\rbrace\subseteq[\omega]^{\omega}$ a decreasing family such that $A_F=\bigwedge_{n<\omega}\omega\times B_n $. So there exists a family $\lbrace B_{n}^m \mid n,m<\omega\rbrace\subseteq[\omega]^\omega$ such that $ B_{n+1}^m \subseteq B_{n}^m$, for all $m,n<\omega$, and $$A=\bigvee_{m<\omega}\bigwedge_{n<\omega}\omega\times B_{n}^m .$$ 

\begin{claim}
Fix $m<\omega$. For all $k<\omega$ there exist $F_{k} \in[T_k ]^{<\omega}$ and $n_k <\omega$ such that $B^{m}_{n_k} \subseteq^* \bigcup_{\sigma\in F_k }A_\sigma$.

\end{claim}

\begin{proof}
Suppose on the contrary that there exists $k<\omega$ such that for all $F\in[T_k]^{<\omega}$ and for all $n<\omega$, $B_{n}^{m}\nsubseteq^* \bigcup_{\sigma\in F}A_\sigma$. This means that there exists $B$, a pseudointersection of the family $\lbrace B^{m}_n \mid n<\omega\rbrace$, which is almost disjoint to every element of the family $\lbrace A_\sigma \mid\sigma\in T_k \rbrace$. Therefore $\omega\times B\subseteq_{\mathcal{NC}}\bigwedge_{n<\omega}\omega\times B_{n}^m$ while being $\mathcal{NC}$-almost disjoint to $\bigvee_{\sigma\in T_k}\omega\times A_\sigma$, which is a contradiction. 

\end{proof}

Fix $m<\omega$. If $n_k$ is the least integer and $F_k$ the smallest finite subset of $T_k$ such that $B^{m}_{n_k} \subseteq^* \bigcup_{\sigma\in F_k }A_\sigma$, it follows that $$T^{m}:=\bigcup_{k<\omega}F_k$$ is a subtree of $T$. 

Unfix $m$. Take $f\in[T]$. Since $A^\prime :=\bigwedge_{n<\omega}\omega\times A_{f\upharpoonright n}\subseteq_{\mathcal{NC}}A$, it follows that there exists $m<\omega$ such that $$ A^\prime \cap \bigwedge_{n<\omega}\omega\times B_{n}^m \in\mathcal{NC}^+ .$$ Take $k<\omega$. If $n_k$ and $F_k$ are as in the previous paragraph, then $f\upharpoonright k\in F_k$, for otherwise we would have that $\vert A_{f\upharpoonright k}\cap B_{n_k}^m \vert<\omega$. Therefore $f\in[T^m ]$ and we conclude that $[T]=\bigcup_{n<\omega}[T_n]$.




\end{proof}

Notice that the existence of said family of finite level subtrees of $T$ gives us the necessary finite steps recipe for approximating the desired substructure. 
  Also observe that if the countable amount of columns of $\omega\times\omega$ gives space to finitely approximating some simple sumprema or infima, like those of Lemma \ref{aI=omega} and Proposition \ref{gdelta,fsigma}, said columns lack ``depth" to approximate some countable but barely more complex algebraic operations. 

In order to use the tower of theorem \ref{aronszjan2} to get a partition of size $\omega_1$ on $\mathcal{P}(\omega\times\omega)\slash\mathcal{NC}$, we need, according to the last proposition, an Aronszajn tree $T$ such that for all $\alpha<\omega_1$ there exists $\lbrace S^{\alpha,m}\mid m<\omega\rbrace$ a family of subtrees of finite levels of $T_\alpha$ such that $$[T_\alpha ]=\bigcup_{m<\omega}[S^{\alpha,m}].$$ The last result of this section shows that this is not the case.

\begin{proposition}
\label{limits_of_aronszjan}
Let $T\subseteq\omega^{<\omega_1}$ be an Aronszajn tree. There exists $\alpha\in Lim(\omega_1 )$ such that if we take $\lbrace S^n \mid n<\omega\rbrace$, a family of finite level subtrees of $T_{<\alpha}$, then $$\bigcup_{n<\omega}[S^n ]\neq[T_{<\alpha}].$$
\end{proposition}

\begin{proof}
First, for $t\in T$ set the notation $t\uparrow:=\lbrace s\in T\mid t\subseteq s\rbrace$. Notice that $$T^\prime :=\lbrace t\in T\mid\vert t\uparrow\vert=\omega_1 \rbrace$$ is a well-pruned Aronszajn tree and that any $\alpha\in Lim(\omega_1)$ proving the proposition for $T^\prime$ will prove it for $T$. So without loss of generality suppose that $T$ is well-pruned. 

\begin{claim}
There exists $\alpha\in Lim(\omega_1 )$ such that for all $t\in T_{<\alpha}$ there exists $dom(t)\leq\gamma<\alpha$ such that $\vert t\uparrow \cap T_\gamma \vert=\omega$. 
\end{claim}
\begin{proof}
Suppose on the contrary that for all $\alpha\in Lim(\omega_1 )$ there exists $t_\alpha \in T_{<\alpha}$ such that $\vert t_\alpha \uparrow\cap T_\gamma \vert<\omega$, for all $dom(t_\alpha )\leq\gamma<\alpha$. Define $g:Lim(\omega_1 )\rightarrow\omega_1$ as $g(\alpha)=dom(t_\alpha )$. This is a regressive function defined on a stationary set. Therefore (see Lemma III.6.14 of \cite{kunen}) there exist $\beta<\omega_1$ and $S\in[\omega_1 ]^{\omega_1 }$ such that $g(\alpha) =\beta$, for all $\alpha\in S$. Furthermore, there exists $t\in T_\beta$ and $S^\prime \in[S]^{\omega_1 }$ such that $t_\alpha =t$ for all $\alpha\in S^\prime$. Since $S^\prime$ is a cofinal subset of $\omega_1$, it follows that $0<\vert t\uparrow\cap T_\gamma \vert<\omega$ for all $dom(t)\leq\gamma<\omega_1$. But a compacity argument gives us a cofinal branch on $T$, which is a contradiction.
\end{proof}

Take the $\alpha$ given by this claim and let $\lbrace S^n \mid n<\omega\rbrace$ be a family of subtrees of $T_{<\alpha}$ with finite levels. Fix $\lbrace\alpha_n \mid n<\omega\rbrace\subseteq\alpha$ an increasing sequence converging to $\alpha$. Begin with $t_0 \in T_{<\alpha}$. The claim gives us $dom(t_0 )\leq\gamma_0 <\alpha$ such that $\vert t_0 \uparrow \cap T_{\gamma_0} \vert=\omega$. Take $t^{\prime}_1 \in(t_0 \uparrow \cap T_{\gamma_0})\setminus S^0$ and choose $t_1 \in t^{\prime}_1 \uparrow\cap T_{\max\lbrace\alpha_0 ,\gamma_0 \rbrace}$. Suppose now that we have $\lbrace t_0 ,...,t_i \rbrace$, a chain in $T_{<\alpha}$, for some $i<\omega$. We know that there exists $dom(t_i)\leq\gamma_i <\alpha$ such that $\vert t_i \uparrow \cap T_{\gamma_i} \vert=\omega$. Therefore we can choose $t_{i+1} \in (t_i \uparrow\cap T_{\max\lbrace\alpha_i ,\gamma_i \rbrace})\setminus S^i$. If $f$ is the cofinal branch of $T_{<\alpha}$ extending the chain $\lbrace t_i \mid i<\omega\rbrace$, by its construction we get that $f\notin\bigcup_{n<\omega}[S^n ]$, whereby proving the proposition.   
\end{proof}

We conclude this section giving some observations and stating a couple of questions. All constructions of towers and partitions on $\mathcal{P}(\omega\times\omega)\slash\mathcal{NC}$ given in this section consist of sets of the type $\omega\times X$, $\bigwedge_{n<\omega}\omega\times X_n$ or  $\bigvee_{n<\omega}\omega\times X_n$ where $\lbrace X\rbrace\cup\lbrace X_n \mid n<\omega\rbrace\subseteq[\omega]^\omega$. Thus far these sets have been proved to be very useful. However they are far from being the only kind of sets of this algebra. For example, if $\lbrace X_n \mid n<\omega\rbrace$ is an independent family, i.e. for all disjoint non-empty $F_0 ,F_1 \in[\omega]^{<\omega}$ the set $$\bigcap_{i\in F_0}X_i \setminus\bigcup_{j\in F_1}X_j$$ is infinite, then $\coprod_{n<\omega}X_n$ is an element of $\mathcal{NC}^+$ not as easily pictured as those of the ``nice" types mentioned above. It is also easy to imagine that there are even less nice elements of $\mathcal{NC}^+$.

Although a good idea for getting a $\mathcal{NC}$-mad family of size $\omega_1$ in ZFC was finally rejected by Proposition \ref{limits_of_aronszjan} the next question remains open:

\begin{question}
\label{omega1_partitions}
Is provable in ZFC that there exists a $\mathcal{NC}$-mad family of size $\omega_1$? Is its existence at least consistent with $\mathfrak{p}>\omega_1$?
\end{question}

It is easy to see that any $\mathcal{NC}$-mad family consisting only of elements of the type $\omega\times X$ or $\bigvee_{n<\omega}\omega\times X_n$ must be of size at least $\mathfrak{a}$. 
Therefore the only hope for getting a $\mathcal{NC}$-mad family of size $\omega_1$ consisting of ``nice" sets, as those discussed in the last paragraphs, while also having $\mathfrak{p}>\omega_1$, can only be achieved with sets of the form $\bigwedge_{n<\omega}\omega\times X_n$. So an auxiliary question to question \ref{omega1_partitions}, translated to a language more familiar to those who deal with infinite combinatorics and very interesting in itself, is the following one. 

\begin{question}
Is it provable in ZFC the existence of a family $\lbrace X_{\alpha}^{n}\mid\alpha<\omega_1 ,~n<\omega\rbrace\subseteq[\omega]^\omega$ such that 

\begin{enumerate}
    \item $X_{\alpha}^{n+1}\subseteq X_{\alpha}^{n}$, for every $\alpha<\omega_1$ and every $n<\omega$,
    
    \item for all $\alpha<\beta<\omega_1$ there exists $n<\omega$ such that $\vert X_{\alpha}^{n}\cap X_{\beta}^{n}\vert<\omega$ and
    
    \item for every $X\in[\omega]^{\omega}$ there exists $\alpha<\omega_1$ such that $\lbrace X\rbrace\cap\lbrace X_{\alpha}^{n}\mid n<\omega\rbrace$ is a centered family?
\end{enumerate}

Is its existence at least consistent with $\mathfrak{p}>\omega_1$?
\end{question}

Notice that if we define $X_\alpha :=\coprod_{n<\omega} X_{\alpha}^{n}$, for all $\alpha<\omega_1$, point 1 says that $X_\alpha$'s are of the ``infimum" type, point 2 says that they are $\mathcal{NC}$-almost disjoint and point 3 gives the madness of the family.

\section{Higher dimensional relatives of the Nowhere Centered ideal}
\label{third_dimension}

As well as $fin^2 :=fin\times fin$ is a ``two-dimensional" version of the ideal $fin$, we can easily generalize it to ideals on $\omega^k$, for all $0<k<\omega$.

\begin{mydef}
\label{fin^k}
For $0<k<\omega$ we will recursively define an ideal $fin^k$ on $\omega^k$ as follows:

\begin{itemize}
    \item $fin^1 :=fin$
    
    \item $fin^k :=\lbrace A\subseteq\omega^k \mid\forall^\infty n<\omega ~A(n)\in fin^{k-1}\rbrace$ for all $k>1$, where $A(n):=\lbrace\overline{x}\in\omega^{k-1}\mid(n)\frown\overline{x}\in A\rbrace$.
\end{itemize}
\end{mydef}

It is easy to see that each $fin^k$ is an ideal and a Borel subset of $\omega^k$. Some kinds of sets that are found in these ideals are the following ones: 

\begin{lemma}
Take $k\geq 1$. Then the following statements hold:

\begin{enumerate}
    \item $\lbrace\overline{x}\in\omega^k \mid\overline{x}(i)<n\rbrace\in fin^k$, for all $i<k$ and $n<\omega$.
    
    \item If $A\subseteq\omega^k$ and $\vert\lbrace n<\omega\mid\overline{x}\frown(n)\in A\rbrace\vert<\omega$, for all $\overline{x}\in\omega^{k-1}$, then $A\in fin^k$.
    
    \item If $1\leq j<k$, $e_0 <...<e_{j-1}<e<k$ and $g:\omega^j \rightarrow\omega$, then $$\lbrace\overline{x}\in\omega^k \mid\overline{x}(e)<g(\overline{x}(e_0 ),...,\overline{x}(e_{j-1}))\rbrace\in fin^k .$$
    
\end{enumerate}

\end{lemma}

\begin{proof}
Since statements 1 and 2 are simple and even provable from Definition \ref{fin^k} and statement 3, only the last one will be proved. It will be inductively proved for $1\leq j<k<\omega$. For the basic case when $j=1$ and $k=2$ it is easy to see that $\lbrace(x,y)\in\omega^2 \mid y<g(x)\rbrace\in fin^2$, for all $g\in\omega^\omega$. 

Now suppose that we have proved it for all $j^\prime <k^\prime $ such that either $j^\prime <j$ or $k^\prime <k$, for some $1\leq j<k<\omega$. Take $e_0 <...<e_{j-1}<e<k$ as well as $g:\omega^j \rightarrow\omega$. If $A_g :=\lbrace\overline{x}\in\omega^k \mid\overline{x}(e)<g(\overline{x}(e_0 ),...,\overline{x}(e_{j-1}))\rbrace$ and $e_0 =0$, then $A_g (n)\in fin^{k-1}$, for all $n<\omega$, by hypothesis induction for $(j-1,k-1)$. If $0<e_0$, equally $A_g (n)\in fin^{k-1}$, for all $n<\omega$, by hypothesis induction for $(j,k-1)$. Either way $A\in fin^k$ and the lemma is proved. 
\end{proof}

In order to relate these paragraphs with the main subject of Section \ref{quotient}, notice that the relation $\subseteq^*$, though not explicit in Definition \ref{nci}, is fundamental in defining the Nowhere Centered ideal. Its relation with the ideal $fin\times fin$ has also been highlighted. Therefore it is not surprising that the ideals $fin^k$ help get higher dimensional relatives of $\mathcal{NC}$. 

\begin{mydef}
 If $k\geq 2$, define $$\mathcal{NC}^k :=\lbrace A\subseteq\omega^k \mid\forall(A_0 ,...,A_{k-2})\in([\omega]^\omega)^{k-1}~ \forall^\infty n<\omega ~\prod_{i<k-1}A_i \nsubseteq_{fin^{k-1}}A(n)\rbrace.$$
\end{mydef}

Observe that $\mathcal{NC}^2 =\mathcal{NC}$ and that $fin^k \subseteq\mathcal{NC}^k$, for all $k<\omega$. Also that $\mathcal{NC}^k$ is a downward closed coanalytic subset of $\mathcal{P}(\omega^k )$, for all $k\geq 2$. 
In order to prove that they are ideals we will need a lemma. 

\begin{lemma}
\label{cubes}
Take $k\geq 1$ and a family $\lbrace A_{n}^0 ,A_{n}^1 \mid n<\omega\rbrace\subseteq (fin^k )^+$. If there exists $(X_0 ,..., X_{k-1})\in([\omega]^\omega )^k$ such that $$\prod_{l<k}X_l \subseteq_{fin^k} A_{n}^0 \cup A_{n}^1 ,$$ for all $n<\omega$, then there exist $(Y_0 ,...,Y_{k-1})\in\prod_{l<k}[X_l]^\omega$, $Z\in[\omega]^\omega$ and $i<2$ such that $$\prod_{l<k}Y_l \subseteq_{fin^k} A_{n}^i ,$$ for all $n\in Z$.
\end{lemma}

\begin{proof}
It will be inductively proved for $k\geq 1$. Suppose that $k=1$ and that we have $\lbrace A_{n}^0 ,A_{n}^1 \mid n<\omega\rbrace\subseteq[\omega]^\omega$ and $X\in[\omega]^\omega$ such that $X\subseteq^* A_{n}^0 \cup A_{n}^1$ for all $n<\omega$. Then any ultrafilter $\mathcal{U}$ containing $X$ will choose $i_n <2$ such that $A_{n}^{i_n}\in\mathcal{U}$, for all $n<\omega$. There exist $Z\in[\omega]^\omega$ and $i<2$ such that $i=i_n$, for all $n\in Z$. Take $Y\in[\omega]^\omega$, a pseudointersection of $\lbrace X\rbrace\cup\lbrace A_{n}^i \mid n\in Z\rbrace$ and the basic case is proved. 

Suppose now that the lemma is proved for some $k\geq 1$ and that we have $\lbrace A_{n}^0 ,A_{n}^1 \mid n<\omega\rbrace\subseteq (fin^{k+1})^+$ and $(X_0 ,..., X_{k})\in([\omega]^\omega )^{k+1}$ such that $$\prod_{l\leq k}X_l \subseteq_{fin^{k+1}} A_{n}^0 \cup A_{n}^1 ,$$ for all $n<\omega$. Observe that for all $n<\omega$ 
$$\prod_{1\leq l\leq k}X_l \subseteq_{fin^k} A_{n}^0 (m)\cup A_{n}^1 (m),$$ for almost all $m\in X_0$.  

Using the inductive hypothesis, we can recursively construct, for $n<\omega$, sequences $\lbrace (Y_{0}^n ,...,Y_{k}^n )\mid n<\omega\rbrace$ and $\langle i_n \mid n<\omega\rangle\in 2^\omega$ such that 

\begin{itemize}
    \item $Y_{l}^0 \in[X_l ]^\omega$, 

    \item $Y_{l}^{n+1}\in[Y_{l}^n ]^\omega$, for all $l\leq k$ and $n<\omega$, and 
    
    \item 
    $$\prod_{1\leq l\leq k}Y_{l}^n \subseteq_{fin^k}A_{n}^{i_n}(m),$$ for all $m\in Y_{0}^n$, for all $n<\omega$. 
\end{itemize}

For $l\leq k$, take $Y_l \in[\omega]^\omega$, a pseudointersection of the family $\lbrace Y_{l} \mid n<\omega\rbrace$. There exist $Z\in[\omega]^\omega$ and $i<2$ such that $i_n =i$, for all $n\in Z$. If $n\in Z$, then $$\prod_{1\leq l\leq k}Y_{l}^n \subseteq_{fin^k}A_{n}^i (m),$$ for almost all $m\in Y_{0}$. Therefore $$\prod_{l\leq k}Y_{l} \subseteq_{fin^{k+1}}A_{n}^i ,$$ for all $n\in Z$. 
\end{proof}









\begin{proposition}
\label{next_level}
For all $k\geq 2$ the family $\mathcal{NC}^k$ is a proper ideal of $\omega^k$.

\end{proposition}

\begin{proof}
Since $\omega^k \notin\mathcal{NC}^k$ and it is downward closed, it only remains to be proved that $A_0 \cup A_1 \in\mathcal{NC}^k$, for all $A_0 ,A_1 \in\mathcal{NC}^k$.  Take $A_0$ and $A_1$ subsets of $\omega^k$ such that $A_0 \cup A_1 \in \mathcal{P}(\omega^k )\setminus\mathcal{NC}^k$. There exists $(A_0 ,..., A_{k-2})\in([\omega]^\omega )^{k-1}$ such that $$\prod_{l<k-1}A_l \subseteq_{fin^k} A_{0}(n)\cup A_{1}(n),$$ for infinitely many $n<\omega$. Applying Lemma \ref{cubes} we get $(A^{\prime}_0 ,...A^{\prime}_{k-2})\in([\omega]^\omega )^{k-1}$ and $i<2$ such that $$\prod_{l<k-1}A^{\prime}_l \subseteq_{fin^k} A_{i}(n),$$ for infinitely many $n<\omega$, which means that $A_i \notin\mathcal{NC}^k$. Therefore, if $A_0 ,A_1 \in\mathcal{NC}^k$, then $A_0 \cup A_1 \in\mathcal{NC}^k$.

\end{proof}



Observe that for all $A\in(\mathcal{NC}^k )^+$ there exists $(X_0 ,...,X_{k-1})\in([\omega]^\omega )^k$ such that $\prod_{i<k}X_i \subseteq_{fin^k}A$. Obviously the set $$\lbrace\omega^{k-1}\times X\mid X\in[\omega]^\omega \rbrace$$ form a subalgebra of $\mathcal{P}(\omega^k)\slash\mathcal{NC}^k$ isomorphic to $\mathcal{P}(\omega)\slash fin$.  All results from Lemma \ref{aI=omega} to Theorem \ref{aronszjan2} have a version on these quotients. But of more importance to this section is how the increase in dimension upgrades the reach of Lemma \ref{aI=omega}. 

\begin{nota}
Take $1\leq j<k<\omega$ and $E\in[k-1]^j$. If $\lbrace A_{\overline{n}}\mid\overline{n}\in\omega^E \rbrace\subseteq \mathcal{P}(\omega)$, define $$\coprod_{\overline{n}\in\omega^E}A_{\overline{n}} :=\lbrace\overline{x}\in\omega^k\mid
\overline{x}(k-1)\in A_{\overline{x}\upharpoonright E}\rbrace.$$ 
\end{nota}

\begin{theorem}
\label{countable_iteration}
Take $1\leq j<k<\omega$ and $l_0 <...<l_{j-1}<k-1$, and define $E:=\lbrace l_i \mid i<j\rbrace$. Then for all family $\lbrace A_{\overline{n}}\mid\overline{n}\in\omega^j \rbrace\subseteq[\omega]^\omega$ the following equality holds in $\mathcal{P}(\omega^k )\slash\mathcal{NC}^k$: 

$$\coprod_{\overline{n}\in\omega^F}(\bigcap_{i_0 \leq\overline{n}(l_0 )}\bigcup_{i_1 \leq\overline{n}(l_1 )}...~ A_{(i_0 ,...,i_{j-1})})=$$

$$=\bigwedge_{n_0 <\omega}\bigvee_{n_1 <\omega}...~\omega^{k-1}\times A_{(n_0 ,...,n_{j-1})}.$$

\end{theorem}

\begin{proof}
It will be inductively proved for $j\leq 1$. Take $k\leq 2$, $l<k-1$ and a family $\lbrace A_n \mid n <\omega\rbrace\subseteq[\omega]^\omega$. Define $$A:=\coprod_{\overline{n}\in\omega^{\lbrace l\rbrace}}\bigcap_{i\leq\overline{n}(l)}A_i.$$ If $n<\omega$, then $A\setminus(\omega^{k-1}\times A_n )\subseteq\lbrace\overline{x}\in\omega^k \mid\overline{x}(l)<n\rbrace$ which is an element of $\mathcal{NC}^k$. Therefore $A\subseteq_{\mathcal{NC}^k}\omega^{k-1}\times A_n$, for all $n<\omega$. Take $B_0 ,...,B_{k-1}\in[\omega]^\omega$ such that $$\prod_{i<k}B_i\subseteq_{fin^k}\omega^{k-1}\times A_n ,$$ for all $n<\omega$. Therefore $B_{k-1}\subseteq^* A_n $, for all $n<\omega$, and hence it follows that $$\prod_{i<k}B_i \setminus A\in fin^k .$$ We conclude that $A=\bigvee_{n<\omega}A_{(n)}$.

Suppose we have proved the theorem for some $j^\prime \geq 1$. For $j=j^\prime +1$, take $k\geq j+1$, $l_0 <...<l_{j-1}<k-1$, and a family $\lbrace A_{\overline{n}}\mid\overline{n}\in\omega^j \rbrace\subseteq[\omega]^\omega$. With $E:=\lbrace l_i \mid i<j\rbrace$. Define $$A:=\coprod_{\overline{n}\in\omega^F}\bigcap_{i_0 \leq\overline{n}(l_0)}\bigcup_{i_1 \leq\overline{n}(l_1)}... ~A_{(i_0 ,...,i_{j-1})}$$ and $$A_n :=\coprod_{\overline{n}\in\omega^{F\setminus\lbrace l_0 \rbrace}}\bigcup_{i_1 \leq\overline{n}(l_1)}\bigcap_{i_2 \leq\overline{n}(l_2)}... ~A_{(n,i_1 ,...,i_{j-1})},$$ for all $n<\omega$. By hypothesis induction $$A_n =\bigvee_{n_1 <\omega}\bigwedge_{n_2 <\omega}...~\omega^{k-1}\times A_{(n,n_1,...,n_{j-1})},$$ for all $n<\omega$. We need to prove that $A=\bigwedge_{n<\omega}A_n$. Take $n<\omega$ and $\overline{x}\in A$ such that $\overline{x}(l_0 )\geq n$. It follows that $$\overline{x}(k-1)\in\bigcup_{i_1 \leq\overline{x}(l_1 )}\bigcap_{i_2 \leq\overline{x}(l_2 )}...~A_{(n,i_1 ,...,i_{j-1})}$$ and hence that $\overline{x}\in A_n$.  We conclude that $A\subseteq_{\mathcal{NC}^k}A_n$, for all $n<\omega$. 

Now take $B_0 ,...,B_{k-1}\in[\omega]^\omega$ such that $$\prod_{i<k}B_i\subseteq_{fin^k}A_n,$$ for all $n<\omega$. Then for all $n_0 <\omega$ we have that $$\prod_{i<k}B_i \subseteq_{fin^k}\bigvee_{n_1 <\omega}\bigwedge_{n_2 <\omega}...~\omega^{k-1}\times A_{(n,n_1,...,n_{j-1})}.$$ Notice that we can construct $g_1 :\omega\rightarrow\omega$ and get $B_{i}^1 \in[B_i ]^\omega$, for $i<k$, such that for all $n_0 <\omega$ we have that $$\prod_{i<k}B_{i}^1 \subseteq_{fin^k}\bigwedge_{n_2 <\omega}\bigvee_{n_3 <\omega}...~\omega^k \times A_{(n_0 ,g_1 (n_0 ),n_2 ,...,n_{j-1})}.$$ Similarly we can get $B_{i}^3$, for $i<k$, and $g_3 :\omega^2 \rightarrow\omega$ such that for all $n_0 <\omega$ and for all $n_2 <\omega$ $$\prod_{i<k}B_{i}^3 \subseteq_{fin^k}\bigwedge_{n_4 <\omega}\bigvee_{n_5 <\omega}...~\omega^k \times A_{(n_0 ,g_1 (n_0 ),n_2 ,g_3 (n_0 ,n_1 ),...,n_{j-1})}.$$  Following this path we finally get $B_{i}^{\prime}\in [B_i ]^\omega$, for $i<k$, and $g_1 :\omega\rightarrow\omega$, $g_3 :\omega^2 \rightarrow\omega$, $g_5 :\omega^3 \rightarrow\omega$,..., for impair numbers $<j$, such that for all $n_0 ,n_2, n_4 ,... <\omega$ (as many variables as pair numbers $<j$), we have that $$\prod_{i<k}B_{i}^\prime \subseteq_{fin^k}\omega^{k-1}\times A_{(n_0 , g_1 (n_0),n_2 ,g_3 (n_0 ,n_2),n_4 ,g_5(n_0 ,n_2 ,n_4),....)}$$ and hence that $$B_{k-1}^\prime \subseteq^* A_{(n_0 , g_1 (n_0)),n_2 ,g_3 (n_0 ,n_2),n_4 ,g_5(n_0 ,n_2 ,n_4),....)}.$$ 

If $C:=\prod_{i<k}B_{i}^\prime \setminus A$, we need to prove that $C\in\mathcal{NC}^k$ to finish the proof. For impair $i<j$ define $h_i :\omega^{\frac{i+1}{2}}\rightarrow\omega$ as follows: $$h_i (n_0, n_1 ,...)=\max\lbrace g_i (i_0 ,i_1 ,...)\mid i_0 \leq n_0 , i_1 \leq n_1 ...\rbrace.$$ Since $$\lbrace\overline{x}\in\omega^k \mid\overline{x}(l_1 )<h_1 (\overline{x}(l_0 ))\vee\overline{x}(l_3 )<h_3 (\overline{x}(l_0 ),\overline{x}(l_2 ))\vee...\rbrace$$ is an element of $\mathcal{NC}^k$, to prove that $C\in\mathcal{NC}^k$ we just need to prove that $$C^\prime :=C\cap\lbrace\overline{x}\in\omega^k \mid\overline{x}(l_1 )\geq h_1 (\overline{x}(l_0 ))\wedge\overline{x}(l_3 )\geq h_3 (\overline{x}(l_0 ),\overline{x}(l_2 ))\wedge...\rbrace$$ lies in $\mathcal{NC}^k$. If $\overline{y}\in\omega^{k-1}$ and $\overline{y}(l_1 )\geq h_1 (\overline{y}(l_0 ))$, $\overline{y}(l_3 )\geq h_3 (\overline{y}(l_0 ),\overline{y}(l_2 ))$ and so forth, then $$\lbrace n<\omega\mid\overline{y}\frown(n)\in C^\prime \rbrace\subseteq B_{k-1}^\prime \setminus\bigcap_{i_0 \leq\overline{y}(l_0 )}\bigcup_{i_1 \leq\overline{y}(l_1)}...~A_{(i_0 ,...,i_{j-1})}$$ which is a finite set. Therefore $C^\prime \in\mathcal{NC}^k$ and we conclude the proof.  

\end{proof}

Theorem \ref{countable_iteration} is powerful enough to settle Question \ref{omega1_partitions} on $\mathcal{P}(\omega^k )\slash\mathcal{NC}^k$, for $k\geq 2$, while it is left open on $\mathcal{P}(\omega^2 )\slash\mathcal{NC}$. Whether these results give some light or cast shadows on the questions of Section \ref{quotient} is not known. 

\begin{theorem}
Let $T\subseteq\omega^{<\omega_1}$ be an Aronszajn tree. Take $\lbrace A_\sigma \mid \sigma\in T\rbrace\subseteq[\omega]^\omega$ such that if $\sigma\subseteq\tau\in T$, then $A_\tau \subseteq^* A_\sigma$ and that the set $\lbrace A_\sigma \mid \sigma\in T_\alpha \rbrace$ is a partition of $\omega$ for all $\alpha <\omega_1$. If $3\leq k<\omega$, then 
\begin{enumerate}
    \item the sets $X_\alpha :=\bigvee_{\sigma \in T_\alpha }\omega^{k-1}\times A_\sigma$, for $\alpha<\omega_1$, form a tower in $\mathcal{P}(\omega^k )\slash\mathcal{NC}^k$ and
    \item the sets $Y_\alpha := (\bigwedge_{\beta <\alpha }X_\beta )\setminus X_\alpha$, for $0 <\alpha <\omega_1$, form a partition in $\mathcal{P}(\omega^k )\slash\mathcal{NC}^k$.
\end{enumerate}
\end{theorem}
\begin{corollary}
$\overline{\mathfrak{a}}(\mathcal{NC}^k )=\omega_1$, for all $3\leq k<\omega$. 
\end{corollary}

From the proof of Theorem \ref{countable_iteration} it is clear how the first $k-1$ dimensions of $\omega^k$ give us room enough to finitely approximate in the last coordinate the desired algebraic operation. It is also clear that the same idea cannot be equally applied to a family indexed by $\omega^k$. If anything, Proposition \ref{gdelta,fsigma} hints to the possibility that there is a limit on the times that $\bigvee$ and $\bigwedge$, with countable indexes, can be iterated on $\mathcal{P}(\omega^k )\slash\mathcal{NC}^k$. 

\begin{conj}
\label{limits_of_dimension}
For all $2\leq k<\omega$, there exists $\lbrace A_{\overline{x}}\mid\overline{x}\in\omega^k \rbrace\subseteq[\omega]^\omega$ such that $$\bigwedge_{n_0 <\omega}\bigvee_{n_1 <\omega}...~\omega^{k-1}\times A_{(n_0 ,n_1 ,...,n_{k-1})}$$ does not exist in $\mathcal{P}(\omega^k )\slash\mathcal{NC}^k$. 
\end{conj}

We finish this section with a question about similar results for $\omega\leq\alpha<\omega_1$. Take $X$ a countable set and an ideal $\mathcal{I}$ such that there exists $\mathcal{A}\leq \mathcal{P}(X)\slash\mathcal{I}$ isomorphic to $\mathcal{P}(\omega)\slash fin$. For $\alpha<\omega_1$, recursively define \begin{itemize}
    \item $\Sigma_{0}^0 (\mathcal{A}):=\mathcal{A}$
    
    \item $\Pi_{\alpha}^0 (\mathcal{A}):=\lbrace A\in \mathcal{P}(X)\slash\mathcal{I}\mid A\in\Sigma_{\alpha}^0 (\mathcal{A})\rbrace$
    
    \item $\Sigma_{\alpha}^0 (\mathcal{A}):=\lbrace A\in \mathcal{P}(X)\slash\mathcal{I}\mid\exists\lbrace A_n \mid n<\omega\rbrace\subseteq\bigcup_{\beta<\alpha}\Pi_{\beta}^0 (\mathcal{A}) ~A=\bigvee_{n<\omega}A_n \rbrace.$
\end{itemize} For $\alpha<\omega_1$, we will say that $\mathcal{A}$ is $\alpha$-$\sigma$-\textit{closed}, if it is $\beta$-$\sigma$-\textit{closed}, for all $\beta<\alpha$, and for all $\lbrace A_n \mid n<\omega\rbrace\subseteq\bigcup_{\beta<\alpha}\Pi_{\beta}^0 (\mathcal{A})$ there exists $A\in \mathcal{P}(X)\slash\mathcal{I}$ such that $A=\bigvee_{n<\omega}A_n$. With this notation, Theorem \ref{countable_iteration} says that $\lbrace\omega^{k-1}\times X\mid X\in[\omega]^\omega \rbrace$ is $k-1$-$\sigma$-closed while Conjecture \ref{limits_of_dimension} says that it is not $k$-$\sigma$-closed, for all $2\leq k<\omega$. 

\begin{question}
Does there exist a sequence $\lbrace\mathcal{I}_{\alpha}\mid\alpha<\omega_1 \rbrace$, of definable ideals on countable sets $X_\alpha$, and a sequence $\lbrace\mathcal{A}_\alpha \mid\alpha<\omega_1 \rbrace$, of subalgebras of $\mathcal{P}(X_\alpha )\slash\mathcal{I}_\alpha$ isomorphic to $\mathcal{P}(\omega)\slash fin$, such that $\mathcal{A}_\alpha$ is $\alpha$-$\sigma$-closed, but not $\alpha+1$-$\sigma$-closed, for all $\alpha<\omega_1$?
\end{question}


\section{The ideal $\mathcal{NC}$ as a substructure of $\mathcal{P}(\omega\times\omega)$ modulo $[\omega\times\omega]^{<\omega}$}
\label{add_cov_non_cof}

In this section we will briefly study $\mathcal{NC}$ as such. Some close ideals and Kat\v{e}tov relations between them will be helpful. 

\begin{mydef}
 The Kat\v{e}tov-Blass order $\leq_{KB}$ on ideals is defined as follows: if $\mathcal{I}$ and $\mathcal{J}$ be two ideals on $\omega$,  then $\mathcal{I}\leq_{KB}\mathcal{J}$ if there exists $f:\omega\rightarrow\omega$ finite-to-one such that $f^{-1}[I]\in\mathcal{J}$ for all $I\in\mathcal{I}$. 
 
 If the function $f$ is not required to be finite-to-one, we get the Kat\v{e}tov relation $\mathcal{I}\leq_K \mathcal{J}$. 
\end{mydef}

\begin{mydef}
 \begin{itemize}
 
         \item $\Delta :=\lbrace(n,m)\in\omega\times\omega\mid m\leq n\rbrace$
         
         \item $\mathcal{ED}_{fin}:=\lbrace X\subseteq\Delta\mid\exists n<\omega~\forall m<\omega~\vert X(m)\vert\leq n\rbrace$
         
        \item A graph is a pair $(X,E)$, where $E\subseteq[X]^2$. We say that $Y\subseteq X$ induces a complete subgraph of $(X,E)$ if $[Y]^2 \subseteq E$. 
         
         \item $\mathcal{G}_c :=\lbrace E\subseteq[\omega]^2 \mid\forall Y\in[\omega]^\omega Y~does~not~induce~a~complete~subgraph\rbrace.$
     \end{itemize}
\end{mydef}

The ideal $\mathcal{ED}_{fin}$ can also be defined as the ideal on $\Delta$ generated by the functions from $\omega$ to $\omega$ bounded by the identity function. 

\begin{proposition}
\label{katetov relations}
The following relations hold: 

\begin{itemize} 

\item $fin\times fin\leq_{KB} \mathcal{NC}$ 

\item $\mathcal{ED}_{fin} \leq_{KB} \mathcal{NC}$

\item $\mathcal{NC}\leq_{KB}\mathcal{G}_c$.
\end{itemize}
\end{proposition}

\begin{proof}
Since $fin\times fin\subseteq\mathcal{NC}$, the Kat\v{e}tov-Blass inequality follows. 

For the second relation take the function $f:\omega\times\omega\rightarrow\Delta$ defined by the rule $$f(n,m)=(\max\lbrace n,m\rbrace,\min\lbrace n,m\rbrace),$$ for all $n,m<\omega$. Let $A\subseteq\Delta$ be a function. The set $f^{-1}[A]$ is clearly the union $$A\cup\lbrace (A(n),n)\mid n<\omega\rbrace.$$ The set $A$ is an element of $fin\times fin$. Notice that the second term of the union is equal to the set $$B:=\coprod_{m<\omega} A^{-1}(m).$$ Since $\lbrace A^{-1}(m)\mid m<\omega\rbrace$ is a pairwise disjoint family, it follows that $B\in\mathcal{NC}$. We conclude that $f$ is a finite-to-one Kat\v{e}tov function.  

For the last inequality consider the function $f:[\omega]^2 \rightarrow\omega\times\omega$ defined by the rule $$f(\lbrace n,m\rbrace)=(\min\lbrace n,m\rbrace,\max\lbrace n,m\rbrace).$$ Take $A\in\mathcal{NC}$ and $X\in[\omega]^\omega$. Since $A\in\mathcal{NC}$, it follows that there exists non-empty $F\in[X]^{<\omega}$ such that $\vert\bigcap_{i\in F}A(i)\vert<\omega$. In particular $\vert\bigcap_{i\in F}A(i)\cap X\vert<\omega$. Take $k\in X\setminus\bigcap_{i\in F}A(i)$ bigger than every element of $F$. Then there exists $i\in F$ such that $k\notin A(i)$. Since $f(\lbrace i,k\rbrace)=(i,k)\notin A$, it follows that $\lbrace i,k\rbrace\notin f^{-1}[A]$. Therefore $X$ does not induce a complete subgraph of $(\omega,f^{-1}[A])$, and we conclude that $f^{-1}[A]\in\mathcal{G}_c$. 
\end{proof}

Both the additivity and the cofinality of the Nowhere Centered ideal are well-known cardinals. In order to prove that its cofinality is equal to $\mathfrak{c}$ we will use the fact that if $\kappa<\mathfrak{c}$ and $\lbrace X_\alpha \mid\alpha<\kappa\rbrace$ is a family of infinite coinfinite subsets of $\omega$, then there is coinfinite $X\in[\omega]^\omega$ such that $X\nsubseteq^* X_\alpha$ for every $\alpha<\kappa.$ Indeed, consider any family $\lbrace A_\beta \mid\beta<\mathfrak{c}\rbrace$ such that $\lbrace\omega\setminus A_\beta \mid\beta<\mathfrak{c}\rbrace$ is some mad family. Suppose that for all $\beta<\mathfrak{c}$ there exists $\alpha<\kappa$ such that $ A_\beta \subseteq^* X_\alpha$. Since $\kappa<\mathfrak{c}$, there exist $\beta_0 <\beta_1 <\mathfrak{c}$ and $\alpha<\kappa$ such that $A_{\beta_0}\cup A_{\beta_1}\subseteq^* X_\alpha$, but this means that $\omega\subseteq^* X_\alpha$, which is a contradiction. 

\begin{theorem}
$add^* (\mathcal{NC})=\omega$ and $cof^* (\mathcal{NC})=\mathfrak{c}$.
\end{theorem}
\begin{proof}
Take the family $\mathcal{A}:=\lbrace\lbrace n\rbrace\times\omega\mid n<\omega\rbrace$ which is a subset of $\mathcal{NC}$. Clearly if $A\subseteq\omega\times\omega$ and $\lbrace n\rbrace\times\omega\subseteq^* A$, for all $n<\omega$, then $A\in\mathcal{NC}^+$. Therefore $\mathcal{A}$ witnesses that $add^* (\mathcal{NC})=\omega$. 

Now take $\kappa <\mathfrak{c}$ and $\mathcal{C}:=\lbrace A_\alpha \mid\alpha <\kappa\rbrace$ a subfamily of $\mathcal{NC}$. Consider the family $K_0 :=\lbrace\alpha <\kappa\mid A_\alpha (0)\nsupseteq^* \omega\rbrace$ and define $C_0 :=\omega$. Since $\kappa<\mathfrak{c}$, there exists $B_0 \subseteq\omega$ such that $\vert B_0 \vert=\vert\omega\setminus B_0 \vert=\omega$ and that $B_0 \nsubseteq^* A_\alpha (0)$ for all $\alpha\in K_0$. Suppose that for some $0<m<\omega$ we have defined a disjoint family $\lbrace B_i \mid i<m\rbrace\subseteq[\omega]^\omega$ such that $$C_m :=\omega\setminus\bigcup_{i<m}B_i$$ is infinite. Take $K_m :=\lbrace\alpha <\kappa\mid A_\alpha (m)\nsupseteq^* C_m \rbrace$. Since $\kappa <\mathfrak{c}$, there exists $B_m \subseteq C_m $ such that $\vert B_m \vert=\vert C_m \setminus B_m \vert=\omega$ and that $B_m \nsubseteq^* A_\alpha (m)$ for all $\alpha\in K_m$. Thus we can recursively construct a centered family $\lbrace C_m \mid m<\omega\rbrace$ and a disjoint family $\lbrace B_m \mid m<\omega\rbrace$ such that $B_m \subseteq C_m$, for all $m<\omega$, and such that $B_m \nsubseteq^* A_\alpha (m)$, for all $\alpha\in K_m$, for all $m<\omega$.

Define $B:=\coprod_{m<\omega}B_m $, which is clearly an element of $\mathcal{NC}$. Take $\alpha <\kappa$. Since $\lbrace C_m \mid m<\omega\rbrace$ is a centered family and $A_\alpha$ is an element of $\mathcal{NC}$, there exists $n<\omega$ such that $\alpha\in K_n$. Then $B_n \nsubseteq^* A_\alpha (n)$, and hence it follows that $B\nsubseteq^* A_\alpha$. Therefore $B$ witnesses that $\mathcal{C}$ is not a cofinal subset of $\mathcal{NC}$ and we conclude that $cof^* (\mathcal{NC})=\mathfrak{c}$.  
\end{proof}

Since the covering and the uniformity of $\mathcal{NC}$ are not as straightforward as the other two cardinal invariants we will use the following result for helping us to bound them.

\begin{lemma}
\label{cov and non inequalities}
 Let $\mathcal{I}$ and $\mathcal{J}$ be ideals on a countable set. If $\mathcal{I}\leq_K \mathcal{J}$, then $cov^* (\mathcal{J})\leq cov^* (\mathcal{I})$ and $non^* (\mathcal{I})\leq non^* (\mathcal{J})$. 
\end{lemma}

\begin{theorem}
\label{covering}
$\min\lbrace\mathfrak{b},\mathfrak{s}\rbrace\leq cov^*(\mathcal{NC})\leq\mathfrak{b}$. 
\end{theorem}

\begin{proof}

Since $\min\lbrace\mathfrak{b,s}\rbrace\leq cov^* (\mathcal{G}_c)$ and $cov^* (fin\times fin)=\mathfrak{b}$ (see Chapter 1 of \cite{davidmeza}), the theorem follows from Proposition \ref{katetov relations} and Lemma \ref{cov and non inequalities}. 
\end{proof}

Recall that a family $\mathcal{F}\subseteq[\omega]^\omega$ is said to be $\omega$-\textit{hitting} if for all countable family $\lbrace X_n \mid n<\omega\rbrace\subseteq[\omega]^\omega$ there exists $X\in\mathcal{F}$ such that $\vert X\cap X_n \vert=\omega$, for all $n<\omega$. 

\begin{theorem}
\label{cov(NC)<b}
It is consistent that $\omega_1 =\mathfrak{s}=cov^* (\mathcal{NC})<\mathfrak{b}=\omega_2$.
\end{theorem}

\begin{proof}
We begin with $V$ a model of ZFC, and the model that will do is $V[G]$ where $G$ is a generic filter on $\mathbb{H}_{\omega_2}$, i.e. the finite support iteration of length $\omega_2$ of Hechler forcing. It is known that $V[G]\models\mathfrak{b}=\omega_2$ and that finite support iterations of Hechler forcing preserve $\omega$-hitting families (see \cite{baumgartner} or \cite{brendle}). It will be enough to prove that  $\mathbb{H}_{\omega_1}$ adds an $\omega$-hitting subfamily of $\mathcal{NC}$ of size $\omega_1$. 

We know that in each step of this iteration a Cohen real is added. A version of Cohen forcing (due to Hechler \cite{hechler}) is the following one: $$\mathbb{P}:=\lbrace p;\omega\times\omega\rightarrow 2\mid dom(p)=n_p \times m_p ~for~some~n_p ,~m_p <\omega\rbrace$$ where $p\leq q$ iff \begin{itemize} 

\item $p\supseteq q$ and

\item $\vert\lbrace j\in n_q \mid \mathcal{P}(j,i)=1\rbrace\vert\leq 1$ for all $i\in m_p \setminus m_q$. 
1
\end{itemize} Since $\mathbb{P}$ is an atomless countable forcing notion, it is equivalent to Cohen forcing. If $H$ is a $(V,\mathbb{P})$-generic filter for some model $M$ of $ZFC$, it follows that $NC_H :=\lbrace (j,i)\mid\exists p\in H \mathcal{P}(j,i)=1\rbrace$ is an element of $\mathcal{NC}\cap V[H]$ which splits all elements of $[\omega\times\omega]^\omega \cap V$. Therefore, after adding $\omega_1$ many Cohen reals, we get in the generic extension an $\omega$-hitting subfamily of $\mathcal{NC}$ of size $\omega_1$. 
\end{proof}


Recall that Mathias forcing is the set $$\mathbb{M}:=\lbrace(s,X)\mid s\in[\omega]^{<\omega}~X\in[\omega]^\omega \rbrace,$$ with $(t,Y)\leq (s,X)$ if $t\supseteq s$, $Y\subseteq X$ and $t\setminus s\subseteq X$. If $G\subseteq\mathbb{M}$ is a generic filter, the Mathias real added by it is $m_G =\bigcup\lbrace s\mid\exists X\in[\omega]^\omega ~(s,X)\in G\rbrace. $ 

While, as noticed in the last proof, Cohen forcing adds a new element of $\mathcal{NC}$ which hits all reals in the base model, observe that Mathias forcing adds an infinite subset of $\omega\times\omega$ which has finite intersection with all elements of $\mathcal{NC}$ in the base model. Let $ \dot f$ be the name of the enumeration of the Mathias real and take $(s,X)\in\mathbb{M}$ and $A\in\mathcal{NC}$. We know there exist $Y\in[X]^\omega$, $n\geq\vert s\vert$ and $g\in\omega^\omega$ such that $A(m)\cap Y\subseteq g(m)$, for all $m\geq n$. If necessary, extend $s$ for having $\vert s\vert =n$. Now take $Z\in[Y]^\omega$ such that for all $i<\omega$ the $i$-th element of $Z$ is strictly  greater than $g(n+i)$. Then $$(s,Z)\Vdash\forall i<\omega ~\dot f(n+i)\in Y\setminus A(n+i)$$ and therefore $(s,Z)$ forces that the intersection of $A$ and $\dot f$ is finite. 

From the inequalities of Theorem \ref{covering} a question remains open:

\begin{question}
Is is consistent that $\min\lbrace\mathfrak{b},\mathfrak{s}\rbrace<cov^* (\mathcal{NC})$?
\end{question}


We conclude the section talking about the uniformity. A family $\mathcal{R}\subseteq[\omega]^\omega$ is called a \textit{hereditarily reaping} if $\mathcal{R}\cap[X]^\omega$ is reaping in $X$, for every $X\in\mathcal{R}$. It is not hard to see that there exists a hereditarily reaping family of size $\mathfrak{r}$. 

\begin{theorem}
\label{cov(M),b<non(NC)<r}
$\max\lbrace cov(\mathcal{M}),\mathfrak{b}\rbrace\leq non^* (\mathcal{NC})\leq\mathfrak{r}$.\footnote{$cov(\mathcal{M})$ is the smallest size of a family of meager sets of the real line whose union is the whole line. Also, if $\kappa<cov(\mathcal{M})$ and $\mathcal{D}$ is a family of dense subsets of (countable) Cohen forcing, then there exists a $\mathcal{D}$-generic filter. For details see \cite{blass}.} 
\end{theorem}

\begin{proof}
First we will prove that $cov(\mathcal{M})\leq non^* (\mathcal{NC})$. Take $\mathcal{F}\subseteq[\omega\times\omega]^\omega $ of size $\kappa<cov(\mathcal{M})$. Consider the version of Cohen forcing $\mathbb{P}$ defined in the proof of Theorem \ref{cov(NC)<b}. Clearly there exists $\mathcal{D}:=\lbrace\mathcal{D}_\alpha \mid\alpha<\kappa\rbrace$, a family of dense subsets of $\mathbb{P}$, such that if $H$ is a $\mathcal{D}$-generic filter, then $NC_H$ infinitely intersects any element of $\mathcal{F}$. Since $\kappa<cov(\mathcal{M})$, there exists such a filter. 

Now suppose that $\kappa<\mathfrak{b}$ and that $\mathcal{F}:=\lbrace A_\alpha \mid\alpha<\kappa\rbrace$ is a family of infinite subsets of $\omega\times\omega$. For all $A\in[\omega\times\omega]^\omega$ there exists either $n<\omega$ such that $\vert(\lbrace n\rbrace\times\omega)\cap A\vert=\omega$ or $f\in\omega^\omega$ such that $\vert A\cap f\vert=\omega$. Therefore w.l.o.g. there exists a partition $K_0 \cup K_1$ of $\kappa$, such that \begin{itemize} \item for all $\alpha\in K_0$ there exists $n_\alpha <\omega$ such that $A_\alpha\subseteq\lbrace n_\alpha \rbrace\times\omega$ and
\item for all $\alpha\in K_1$ the set $A_\alpha$ is a partial function from $\omega$ to $\omega$. 
\end{itemize}  Since $\kappa<\mathfrak{b}$, there exists $f\in\omega^\omega$ such that $A_\alpha \leq^* f$, for all $\alpha\in K_1$. Therefore $B_1 :=\lbrace (n,m)\mid m\leq f(n)\rbrace$ is an element of $\mathcal{NC}$ such that $B_1 \cap A_\alpha$ is infinite, for all $\alpha\in K_1$. Recall that $\mathfrak{b}\leq\mathfrak{r}$. The set $\lbrace A_\alpha(n_\alpha )\mid\alpha\in K_0\rbrace$ is split by some real $C_0 \in[\omega]^\omega$. Suppose that for some $n<\omega$ we have defined a disjoint family $C_0 , ..., C_n \in[\omega]^\omega$ such that $D_n :=\bigcup_{i\leq n}C_i $ splits $A_\alpha (n_\alpha )$, for all $\alpha\in K_0$.  The family $$\lbrace A_\alpha (n_\alpha )\setminus D_n \mid\alpha\in K_0 \rbrace$$ is split by some $C_{n+1}\in[\omega\setminus D_n]^\omega$. It follows that both $\bigcup_{i\leq n+1} C_i$ and $C_{n+1}$ split $A_\alpha (n_\alpha )$, for all $\alpha\in K_0$. Thus we can recursively construct a disjoint family $\lbrace C_n \mid n<\omega\rbrace\subseteq[\omega]^\omega$ such that $C_n$ splits $A_\alpha (n_\alpha )$, for all $n<\omega$ and all $\alpha\in K_0$. Take $$B_0 :=\coprod_{n<\omega}C_n ,$$ which is an element of $\mathcal{NC}$, and $\alpha\in K_0$.  Since $C_{n_\alpha}$ splits $A_{\alpha}(n_\alpha )$, it follows that $B_0 \cap A_\alpha$ is infinite. Therefore $B:=B_0 \cup B_1$ is an element of $\mathcal{NC}$ such that $B\cap A_\alpha$ is infinite, for all $\alpha<\kappa$. We conclude that $\mathfrak{b}\leq non^* (\mathcal{NC})$. 

Now we prove that $non^* (\mathcal{NC})\leq\mathfrak{r}$. Let $\mathcal{R}$ be a hereditarily reaping family of size $\mathfrak{r}$. We claim that $$\lbrace\lbrace n\rbrace\times X\mid X\in\mathcal{R},~n<\omega\rbrace$$ witnesses that $non^* (\mathcal{NC})\leq\mathfrak{r}$. Suppose that $A\in[\omega\times\omega]^\omega$ is such that $A\cap(\lbrace n\rbrace\times X)$ is infinite, for all $n<\omega$ and all $X\in\mathcal{R}$. Take $X \in\mathcal{R}$. Since $A\cap(\lbrace 0\rbrace\times X)$ is infinite there exists $X_0 \in[X]^\omega \cap\mathcal{R}$ such that $X_0 \subseteq^* A(0)$ or $X_0 \cap A(0) $ is finite. Since this last case cannot happen, we conclude that $X_0 \subseteq^* A(0)$. Suppose that for some $n<\omega$ we have defined a decreasing family $X_0 ,..., X_n$ of elements of $\mathcal{R}$ such that $X_{i}\subseteq^* A(i) $, for all $i\leq n$. Since $A\cap(\lbrace n+1\rbrace\times X_n )$ is infinite, there exists $X_{n+1} \in[X_n ]^\omega \cap\mathcal{R}$ such that $X_{n+1} \subseteq^* A(n+1)$ (for the case when $X_{n+1}\cap A(n+1)$ is finite cannot happen). So we can recursively construct a decreasing subfamily of $\mathcal{R}$ which witnesses that $\lbrace A(n)\mid n<\omega\rbrace$ is a centered family and that $A\notin\mathcal{NC}$. We conclude that $non^* (\mathcal{NC})\leq\mathfrak{r}$.

\end{proof}





In the realm of consistency we have the following result. 

\begin{theorem}
It is consistent that $\max\lbrace cov(\mathcal{M}),\mathfrak{b}\rbrace<non^* (\mathcal{NC}).$
\end{theorem}

\begin{proof}
In Theorem 1.6.12 of \cite{davidmeza} we have a model of $cof(\mathcal{M})=\omega_1$\footnote{$cof(\mathcal{M})$ is the smallest size of a family $\mathcal{F}$ of meager subsets of the real line such that for all meager set $X$ there exists $Y\in\mathcal{F}$ such that $X\subseteq Y$.} and $non^* (\mathcal{ED}_{fin})>\omega_1$. From Proposition \ref{katetov relations} it follows that in this model also $non^* (\mathcal{NC})>\omega_1$ holds. Since $cov(\mathcal{M}),\mathfrak{b}\leq cof(\mathcal{M})$, the inequality holds in this model. 
\end{proof}

Concerning the uniformity of the Nowhere Centered ideal a question remains open: 

\begin{question}
Is it consistent that $non^* (\mathcal{NC})<\mathfrak{r}$?
\end{question}

\end{document}